\documentclass[11pt]{article}

\usepackage[margin=1in]{geometry}
\usepackage{amsmath,amssymb,amsthm,mathtools}
\usepackage[T1]{fontenc}
\usepackage{lmodern}
\usepackage{microtype}
\usepackage[colorlinks=true,linkcolor=blue,citecolor=blue,urlcolor=blue]{hyperref}
\usepackage{braket}

\newtheorem{theorem}{Theorem}
\newtheorem{lemma}{Lemma}
\newtheorem{remark}{Remark}

\DeclareMathOperator{\ra}{\rangle}
\DeclareMathOperator{\la}{\langle}
\DeclareMathOperator{\tr}{trace}
\DeclareMathOperator{\Var}{Var}
\DeclareMathOperator{\Inf}{Inf}

\newcommand{\norm}[1]{\left\lVert #1\right\rVert}
\newcommand{\Hq}{\mathbb H}

\newcommand{\un}{\mathbf{1}}

\newcommand{\Id}{\mathbf 1}

\newcommand{\If}{\mathsf{I}}
\newcommand{\Hf}{\mathsf{H}}
\newcommand{\hbin}{h_2}
\newcommand{\proj}[2]{\lvert #1\rangle\!\langle #2\rvert}

\newcommand{\wt}{\operatorname{wt}}

\newcommand{\Tr}{\operatorname{Tr}}

\newcommand{\bE}{\mathbf E}

\newcommand{\op}{\mathrm{op}}
\newcommand{\ps}{\mathrm{ps}}
\newcommand{\HS}{\mathrm{HS}}
\newcommand{\ketbra}[2]{\ket{#1}\!\!\bra{#2}}
\newcommand{\ee}{\mathbf e}

\newcommand{\bP}{\mathbf P}

\theoremstyle{plain}
\newtheorem{corollary}[theorem]{Corollary}
\newtheorem{proposition}[theorem]{Proposition}

\theoremstyle{definition}

\theoremstyle{remark}

\newtheorem{case[theorem]}{Case}

\newcommand{\eps}{\varepsilon}

\newcommand{\bT}{\mathbb{T}}

\newcommand{\bC}{\mathbb{C}}

\newcommand{\bR}{\mathbb{R}}

{\end{list}}

\begin{document}

\title{Tightness of and counterexamples to several quantum estimates}

\author{Joseph Slote, Alexander Volberg, Haonan Zhang}
\date{}

\maketitle

\begin{abstract}
We prove here several tightness results for such quantum inequalities as the comparison of operator norm and product norm of $d$-local hamiltonians, Bohnenblust--Hille inequality for $d$-local hamiltonians and for quantum Fourier entropy-influence conjecture, we also discuss the quantum Aaronson--Ambainis conjecture in a special case of anti-commuting Pauli strings.
\end{abstract}

\tableofcontents 

\section{Introduction}
\label{intro}

Several recent papers considered harmonic analysis estimates in quantum world.
Among those we can list the Bohnenblust--Hille (BH) inequality in its Hamming cube form due to \cite{DMP}, a celebrated KKL inequality, entropy-influence inequality, and the last but not the least the quantum analog of Aaronson--Ambainis conjecture \cite{AA}. 

\medskip

BH inequality on Hamming cube played a crucial part in tight estimates of optimal number of queries in PAC(=probably approximately correct) algorithm of learning function of a fixed degree $d$ on Hamming cube $\{-1,1\}^N$. It was used by  Eskenazis--Ivanisvili \cite{EI} to find this optimal number of queries.

The story of commutative Bohnenblust--Hille inequality can be traced back to its simplest version: Littlewood's $4/3$ lemma, that claims the following: let $\{a_{ij}\}$ be $N\times N$ matrix and let us know that
$\max_{\eps_i =\pm 1}\max_{\delta_j =\pm 1} |\sum_{i, j=1}^N a_{ij} \eps_i \delta_j|\le 1$; then $\sum_{i, j=1}^N |a_{ij} |^{4/3} \le C_0$, where an absolute finite constant $C_0$ does not depend on $N$.
Constant $C_0=\sqrt{2}$ was found by Szarek approximately 50 years later. Now, if $\delta_i=\eps_i$ for all $i$ this is still a true claim (the sharp $C_0$ might not be known), and the claim becomes what is called now Bohnenblust--Hille inequality on Hamming cube for degree $2$ homogeneous polynomials.

The main feature of all that, and of almost all that follows is that {\bf dimension free} constants are involved.

\medskip

Bohnenblust and Hille's paper \cite{BH} appeared in 1931 and was devoted to {\bf dimension free}  estimates of coefficients of degree $d$ analytic polynomial of $N$ variables:
$$
z=(z_1,\dots, z_n), \, f(z) = \sum_\alpha c_\alpha z^\alpha\Rightarrow \Big (\sum_\alpha |c_\alpha|^{\frac{2d}{d+1}}\Big)^{\frac{d+1}{2d}} \le C(d) \|f\|_{L^\infty(\bT^N)}\,.
$$
Here $\alpha=(\alpha_1,\dots, \alpha_N)$ is multi-index, $|\alpha|=\alpha_1+\dots+\alpha_N$, and $\max |\alpha|\le d$.

It is known that exponent $\frac{2d}{d+1}$ cannot be made smaller in general \cite{D}, Sections 7.3, 7.4. See below a new proof.

\medskip

This inequality for function on multi-torus $\bT^N$ solved a conjecture of Harold Bohr from the theory of Dirichlet series (Dirichlet series are tightly connected with analytic functions of very large (or infinite) number of complex variables), see \cite{D}.

\bigskip

In 80 years  the article A. Defant, L. Frerick, J. Ortega-Cerd\`a, M. Ounaies, and K. Seip \cite{DFOOS}  revisited Bohnenblust and Hille result and improved the constant $C(d)$: in the original paper it was exponential in $d$, and in this paper it was subexponential:
$$
C(d) \le e^{c \sqrt{d\log d}}\,.
$$

The struggle to improve constant is still going on, and one of the reason for that turned out to be tightly connected with quantum algorithms. 
But first the paper A. Defant, M. Mastylo, and A. P\'erez \cite{DMP} proved a discrete result: now polynomials were $f(z) = \sum_\alpha c_\alpha z^\alpha$ as before but
$\alpha_i=0$ or $1$, $i=1,\dots, N$, and $z_i =-1$ or $1$. So, these are polynomials on Hamming cube $\{-1,1\}^N$.

Amazingly the constant $C(d)$ again turned out to be $\le e^{c \sqrt{d\log d}}$. 

\medskip

In the theory of quantum algorithms the arguably most famous open questions is Aaronson--Ambainis problem. We will  explain it below, now we  will only mention that it concerns again degree $d$ polynomials on Hamming cube of very large dimension $N$. If it would be solved that would give
the theoretical understanding of speed-up of quantum algorithms versus random classical algorithms for an important class of problems (including the factorization problem). Improvement of the Bohnenblust--Hille (BH) constant on Hamming cube to $d^C$ would solve Aaronson--Ambainis problem for an interesting sub-class of polynomials.

\bigskip

All this is still open, {\bf especially the constants}. But, as it has been already mentioned, the paper A. Defant, M. Mastylo, and A. P\'erez \cite{DMP} got a brilliant application in \cite{EI} for  another question of Theoretical Computer Science (TCS): optimal number of queries in PAC(=probably approximately correct) learning algorithms. Here $f$, $\deg f\le d$, on Hamming cube $\{-1,1\}^N$ is given in the sense that oracle knows it. The task is to design a (random) algorithm of queries to oracle with minimal number of queries to define this function with error $\le \eps$ with probability $\ge 1-\delta$.

\bigskip

After \cite{EI} it was natural for \cite{RWZ} to ask for non-commutative BH inequality and its possible application for hamiltonians formed by linear combination of Pauli strings of weight at most $d$.

The thing is that Hamming cube and polynomials on it have non-commutative version. However, the Bohnenblast--Hille  inequality for tensor polynomials of Pauli matrices or Heisenberg--Weyl matrices (qubit version or q-dit version), was a bit enigmatic at first. The qubit (Pauli) version was simultaneously solved by Volberg--Zhang and also in the following papers:  Chen--Huang--Preskill \cite{CHP}, Huang--Kueng--Preskill \cite{HKP}.

\subsection{Estimate of ground state of $d$-local Hamiltonians from below \\ by non-commutative Bohnenblust--Hille inequality}
\label{below}
The proof in \cite{VZ} was a  reduction of non-commutative qubit case to commutative Hamming cube $\{-1,1\}^N$ case.
The non-commutative BH inequality  considers the matrices (hamiltonians) 
$$
H= \sum_\alpha c_\alpha \sigma^\alpha,
$$
where $c_\alpha $ are just complex (or better real)  coefficients, $\alpha$ is a multi-index $\alpha=(\alpha_1,\dots, \alpha_n)$, where each $\alpha_j$ is in $\{0,1,2,3\}$ and $\sigma_{\alpha_j}$ is 
one of  Pauli matrices $I, X, Y, Z$, where $\sigma_0=I$ is $2\times 2$ identity and $\sigma_1=X, \sigma_2=Y, \sigma_3=Z$ are $2\times 2$ non-trivial Pauli matrices. The symbol $\sigma^\alpha$ means  the tensor product of corresponding $\sigma_{\alpha_j}$.

\medskip

As we already discussed that sharp constant $BH^d_{\pm}$ of \cite{DMP}  is unknown. What \cite{DMP} established is 
\begin{equation}
\label{ncCd}
BH^d_{\pm}\le e^{c \sqrt{d\log d}}\,.
\end{equation}

In \cite{VZ} and in \cite{BSVZ} it has been proved that
\begin{equation}
\label{BHnc}
\Big (\sum_\alpha |c_\alpha|^{\frac{2d}{d+1}}\Big)^{\frac{d+1}{2d}} \le C(d) \|H\|_{op},
\end{equation}
where 
\begin{equation}
\label{ncCd}
BH^d_{M(\bC^2)}:=C(d) \le 3^{d/2}\cdot BH^d_{\pm}\,.
\end{equation}

We repeat this proof below. The estimate is worse than in \eqref{ncCd}, here it is exponential, there it is subexponential (and unknown).

\medskip

Amazingly, in quantum case it is proved in \cite{S} that it must be exponential in $d$. We repeat the proof below.

\subsection{Estimate of ground state of $d$-local Hamiltonians from above}
\label{above}

It is interesting to be able to find the operator norm of a hamiltonian rapidly. Hamiltonians act on $n$-tensor product of $\bC^2$. There is an easier form of norm, the product norm:
\begin{equation}
\label{prod}
\|H\|_{prod} = \sup_{E} \braket{HE,E},
\end{equation}
where $E=e_1\otimes e_2\otimes \dots \otimes e_n$ is a tensor product of unit vectors $e_i\in \bC^2$. Obviously
$$
\|H\|_{prod} \le \|H\|_{op}\,.
$$
It turns out that for homogeneous degree $d$ hamiltonians:
\begin{equation}
\label{opprodh}
\|H\|_{op} \le 3^d\|H\|_{prod} \,.
\end{equation}

Moreover,  dimension-free (independent of dimension $n$) converse estimate exists for all degree $d$ hamiltonians, \cite{BSVZ}:

\begin{equation}
\label{opprod}
\|H\|_{op} \le (3+3\sqrt{2})^d\|H\|_{prod} \,.
\end{equation}

We also mention that  if $H$ involves only two types of nontrivial Pauli matrices (for example only $X, Y$ but not $Z$) then the estimate is better for homogeneous degree $d$ polynomials:
\begin{equation}
\label{opprod2}
\|H\|_{op} \le 2^d\|H\|_{prod} \,.
\end{equation}

For homogeneous degree $d$ polynomials the constant in \eqref{opprod}  is slightly better, it is $3^d$. Again the story of {\bf constants} is quite interesting. For $d=2$ and homogeneous degree $2$ hamiltonians the constant $9$ was found by Elliot Lieb \cite{L}. Many years have passed before the same constant $9$ was proved for all degree $2$ traceless hamiltonians in \cite{BGKT}. The passage from homogeneous to general was tricky if one wants to keep the same dimension-free constant $9$.

Already for $d \ge 3$ this trick does not work, this explains a strange constant in \eqref{opprod}.

\medskip

An obvious question: must the constant be exponential in $d$? We prove below that yes, it must be. It is maybe worthwhile to mention that the construction of tightness
has a non-trivial relation to the refutation of Einstein--Podolsky--Rosen (EPR) paradox, \cite{EPR}, \cite{Mermin1990}.

\section{The homogeneous product-norm comparison problem}

Let
\[
X=\begin{pmatrix}0&1\\1&0\end{pmatrix},\qquad
Y=\begin{pmatrix}0&-i\\i&0\end{pmatrix},\qquad
Z=\begin{pmatrix}1&0\\0&-1\end{pmatrix}.
\]
A Pauli monomial on $n$ qubits is a tensor product of matrices from
$\{I,X,Y,Z\}$.  Its weight is the number of nonidentity tensor factors.  An
operator is called $d$-homogeneous if every Pauli monomial with nonzero
coefficient has weight exactly $d$.


For a Hermitian operator $H$ on $n$ qubits, define its product-state numerical
radius by
\begin{equation}
\|H\|_{prod}
:=
\sup_{\norm{e_r}=1}
\left|
\left\langle \bigotimes_{r=1}^n e_r,
A\bigotimes_{r=1}^n e_r
\right\rangle
\right|.
\label{eq:ps}
\end{equation}
We use the homogeneous comparison constant
\begin{equation}
\Gamma_d^{\mathrm{hom}}
:=
\sup_{n\geq d}
\sup_{\substack{0\neq H=H^\dagger\\H\text{ is }d\text{-homogeneous}}}
\frac{\norm{H}_{op}}{\norm{H}_{prod}}.
\label{eq:Gamma}
\end{equation}
The known product-state upper estimate gives
\begin{equation}
\Gamma_d^{\mathrm{hom}}\leq 3^d;
\label{eq:known-upper}
\end{equation}
indeed, it is enough in that estimate to maximize over tensor products of Pauli
eigenstates Theorem~1 of \cite{BSVZ}.

The $X,Y$ tensor Rudin--Shapiro construction below gives a degree-$d$ example with ratio
$2^{d-1}$, hence lower exponential base $2$.  


\subsection{A version of Quantum Rudin--Shapiro polynomials}
\label{RSp}

Rudin--Shapiro polynomials were used to build an example of $2\pi$ periodic function of class $\text{Lip}_{1/2}$ but whose Fourier series is not absolutely convergent.
Bernstein's theorem claims that as soon as function is slightly more smooth, that is lies in $\text{Lip}_{1/2+\epsilon}$ then its Fourier series  converges absolutely.

\bigskip

Now we will consider a certain analog but in quantum world.

\medskip

Consider the recursion
$$
P_0=Y, \,\,Q_0=X
$$
\begin{align}
\label{PnQn}
&P_n= X\otimes P_{n-1}+Y\otimes Q_{n-1},
\\
&Q_n = X\otimes Q_{n-1}- Y\otimes  P_{n-1}\,.
\end{align}

One can write it in complex form
\begin{equation}
\label{com}
P_n+i Q_n = (X-iY)\otimes (P_{n-1} +i Q_{n-1})\,.
\end{equation}
Then
$$
\deg P_n=\deg Q_n =n+1,\,\, \|P_n\|_{HS} = \|Q_n\|_{HS}  = \sqrt{2}^n\,.
$$
(As always we consider normalized trace, and therefore normalized Hilbert--Schmidt norm.)

Therefore,
\begin{equation}
\label{opnorm}
\|P_n\|_{op} \ge 2^{n/2}, \quad \|Q_n\|_{op}  \ge 2^{n/2}\,.
\end{equation}

\bigskip

Now let us look at $\|P_n\|_{prod}$,  $\|Q_n\|_{prod}$.

Let $E^n = e_n\otimes E^{n-1}= e_1\otimes e_{2} \otimes\dots\otimes e_{n+1}$.

\medskip

We will use below the Bloch coordinates for unit vector $e_j, j=1,\dots, n+1$, namely $a_x^j:=\la X e_j, e_j\ra, a_y^j\la Y e_j, e_j\ra, a_z^j\la Z e_j, e_j\ra$. We know that
$$
(a_x^j)^2 +(a_y^j)^2 +(a_z^j)^2 =1\,.
$$

Put $p_n:= \la P_n E^n, E^n \ra , \,\, q_n := \la q_n E^n, E^n \ra $.
Then

$$
p_n:= \la P_n E^n, E^n \ra  = p_{n-1} a_x^1 + q_{n-1} a_y^1,\quad  q_n:= \la Q_n E^n, E^n \ra  = q_{n-1} a_x^1 - p_{n-1} a_y^1\,.
$$

Denote
$$
F_n =\begin{bmatrix} \!\!\!\!p_n\,\,\,\, \,\,q_n\\
q_n,\,\, -p_n
\end{bmatrix}
$$

Then of course
\begin{equation}
\label{F}
F_n = F_{n-1} \cdot \begin{bmatrix} a_x^1, -a_y^1\\
a_y^1, \,\,\,a_x^1,
\end{bmatrix}
\end{equation}
where $a_x^1, a_y^1$ are Bloch coordinates of pure state $e_1\otimes e_1$.

Let us write 
$$
\begin{bmatrix} a_x^1, -a_y^1\\
a_y^1, \,\,\,a_x^1\end{bmatrix}=\sqrt{(a_x^1)^2 + (a_y^1)^2} \cdot U^{(1)},
$$
where one can notice two things: 1) $U^{(1)}$ is unitary, 2) $b_1:=\sqrt{(a_x^1)^2 + (a_y^1)^2} \le 1$.

Let $b_j:=\sqrt{(a_x^j)^2 + (a_y^j)^2} $, $j=1,\dots, n+1$ and
$$
\begin{bmatrix} a_x^j, -a_y^j\\
a_y^j, \,\,\,a_x^j\end{bmatrix}=: \sqrt{(a_x^j)^2 + (a_y^j)^2} \cdot U^{(j)} =b_j\cdot  U^{(j)} ,
$$
where $ U^{(j)} $ is unitary, $b_k\le 1$,.

\bigskip

Using these  observations 1), 2) above we can iterate \eqref{F}:
$$
\begin{bmatrix} \!\!p_n,\,\, \,q_n\\
q_n,\,\,\, -p_n
\end{bmatrix}=F_n= b_1 F_{n-1} U^{(1)}= b_1b_2 F_{n-2} U^{(2)} U^{(1)}=\dots
$$

\bigskip

So, 
$$
|p_n|= |\la P_n e_1\otimes e_2\otimes\dots,  e_1\otimes e_2\otimes\dots\ra|\le 1\Rightarrow \|P_n\|_{prod} \le 1,
$$
but $\|P_n\|_{op}\ge 2^{n/2}$ as we saw above. (It can be $2^n$ but I do not see it immediately.)

\medskip

So product norm is exponentially (in degree) smaller than operator norm.

We will now reconcile this example with the example of \cite{Mermin1990} that elaborates on idea of  \cite{GHZ} devoted to Einstein--Podolsky--Rosen paradox \cite{EPR}.

\subsection{Computing operator norm of $P_n$}
\label{opnorm}

Consider
$$
\tilde Q_n:= \frac12\big(( X-iY)^{\otimes n} + (X+iY)^{\otimes n}\big)\,.
$$
$$
\tilde P_n:= -\frac1{2i}\big(( X-iY)^{\otimes n} - (X+iY)^{\otimes n}\big)\,.
$$

Then 
\begin{align}
&2\tilde Q_n  = X \otimes [ ( X-iY)^{\otimes n-1} +  (X+iY)^{\otimes n-1} ]  - iY \otimes [( X-iY)^{\otimes n-1} -  (X+iY)^{\otimes n-1} ]=
\\
& X\otimes 2\tilde P_n - Y\otimes 2\tilde Q_n\,.
\end{align}
\begin{align}
&2\tilde P_n  = iX \otimes [ ( X-iY)^{\otimes n-1} - (X+iY)^{\otimes n-1} ]  +Y \otimes [( X+iY)^{\otimes n-1} +  (X+iY)^{\otimes n-1} ]=
\\
& X\otimes 2\tilde P_n + Y\otimes 2\tilde Q_n\,.
\end{align}
And $\tilde P_0=Y, \tilde Q_0=X$.  Hence,
\begin{equation}
\label{theSame}
\tilde P_n= P_n, \,\, \tilde Q_n =Q_n\,.
\end{equation}

So the tensor Rudin--Shapiro polynomials are Mermin's polynomials from \cite{Mermin1990}.

\medskip

The nice thing is that in their tilde form one can easily compute their  operator norms.

\medskip

In fact, one can write using the form of $X\pm i Y$)
$$
P_1 = \frac12\big(|0\ra \la 1 | +  |1\ra \la 0 | \big)\,.
$$
And thus
$$
P_n = 2^{n-1}\big(|0^{\otimes n}\ra \la 1^{\otimes n}|  +  |1^{\otimes n}\ra \la 0^{\otimes n}|  \big)\,.
$$

The eigenvector with eigenvalue $2^{n-1}$ is just (by a direct simple calculation)
\begin{equation}
\label{Phi}
\Phi_n := \frac1{\sqrt{2}} \big( |0^{\otimes n}\ra + |1^{\otimes n}\ra \big)\,.
\end{equation}

We finally conclude that
\begin{equation}
\label{norm}
\|P_n\|_{op}= 2^{n-1}\,.
\end{equation}
The same is valid for $Q_n$.

\bigskip

It is interesting to notice that $\Phi_3$ was used in Daniel M. Greenberger, Michael A. Horne, and Anton
Zeilinger paper \cite{GHZ} to elegantly solve Einstein--Podolsky--Rosen paradox of \cite{EPR}.

\section{Improving the constant}
\label{conclusion}

In our paper \cite{BSVZ} we proved the estimate for degree $d$ homogeneous hamiltonians in $X, Y, Z$ of the type
$$
\|H\|_{op} \le 3^d \|H\|_{prod},
$$
where $d$ is the degree, so the number of live qubits in each monomial and this estimate absolutely does not depend on the number $n$ of qubits.

In the case $d=2$ homogeneous hamiltonians this was proved by Lieb in \cite{L}, the constant then is $9$. It is still $9$ for all hamiltonians of degree $2$, see \cite{BGKT}.

\medskip

What about the sharpness of $3^{degree}$ constant? Notice that our proof in \cite{BSVZ} would give the estimate
$$
\|H\|_{op} \le 2^d \|H\|_{prod},
$$
if Hamiltonian has only $X, Y$  (and no $Z$ ) involved.

\medskip

What we just proved above is that for such hamiltonians $2^d$ is {\bf very tight}. In fact we gave the example--$P_d$--for which 
the estimate is exactly with constant $2^{d-1}$, where $d$ is the degree (not the number of qubits, the latter can be arbitrary).

\subsection{A quaternion-type construction that improves the constant}

Our aim is to use all three Bloch
coordinates while retaining an exact norm-preserving transfer rule.

Let $\Hq$ denote the real quaternion algebra with basis
$1,\mathbf i,\mathbf j,\mathbf k$ and multiplication rules
\[
\mathbf i^2=\mathbf j^2=\mathbf k^2=-1,
\qquad
\mathbf i\mathbf j=\mathbf k,
\quad
\mathbf j\mathbf k=\mathbf i,
\quad
\mathbf k\mathbf i=\mathbf j.
\]
Write
\[
\ee_1=\mathbf i,\qquad
\ee_2=\mathbf j,\qquad
\ee_3=\mathbf k,
\qquad
\sigma_1=X,\ \sigma_2=Y,\ \sigma_3=Z.
\]
Below $X_r, Y_r, Z_r$ are $n$-tensor products where on each qubit one has identity matrix except the $r$-th qubit, where one has $X, Y, Z$ correspondingly.

\medskip

On $n$ qubits form the quaternion-valued operator
\begin{equation}
\mathcal R_d
:=
\prod_{r=1}^d
\left(X_r\mathbf i+Y_r\mathbf j+Z_r\mathbf k\right)
=
H_d^{(0)}+H_d^{(1)}\mathbf i+H_d^{(2)}\mathbf j+H_d^{(3)}\mathbf k,
\label{eq:Rdef}
\end{equation}
where the factors are multiplied in increasing qubit order indicated by $r$.  Each component
$H_d^{(\mu)}$ is a real linear combination of full-weight Pauli monomials and is
therefore Hermitian and $d$-homogeneous.

This was the recursion
$$
\mathcal R_d =\mathcal R_{d-1}\cdot \left(X_d\mathbf i+Y_d\mathbf j+Z_d\mathbf k\right)
$$
not unlike the one in \eqref{com}.

 
 Equivalently the above recursion $\mathcal R_d =\mathcal R_{d-1}\cdot \left(X_d\mathbf i+Y_d\mathbf j+Z_d\mathbf k\right)$ with initial values
\[
H_0^{(0)}=I,
\qquad
H_0^{(1)}=H_0^{(2)}=H_0^{(3)}=0,
\]
can be written in coordinate form
\begin{align}
H_r^{(0)}
&=-H_{r-1}^{(1)}\cdot X_r
  -H_{r-1}^{(2)}\cdot Y_r
  -H_{r-1}^{(3)}\cdot Z_r,
\label{eq:rec0}\\
H_r^{(1)}
&= H_{r-1}^{(0)}\cdot X_r
  +H_{r-1}^{(2)}\cdot Z_r
  -H_{r-1}^{(3)}\cdot Y_r,
\label{eq:rec1}\\
H_r^{(2)}
&= H_{r-1}^{(0)}\cdot Y_r
  -H_{r-1}^{(1)}\cdot Z_r
  +H_{r-1}^{(3)}\cdot X_r,
\label{eq:rec2}\\
H_r^{(3)}
&= H_{r-1}^{(0)}\cdot Z
  +H_{r-1}^{(1)}\cdot Y
  -H_{r-1}^{(2)}\cdot X.
\label{eq:rec3}
\end{align}
The first scalar component is the antiferromagnetic Heisenberg interaction
\begin{equation}
H_2^{(0)}=-(X\otimes X+Y\otimes Y+Z\otimes Z)\otimes I\otimes\dots\otimes I\,.
\label{eq:H2}
\end{equation}
which  (in its essential, first two first qubit part) is
$$
H_2^{(0)} = I\otimes I-2\,\text{SWAP} \,.
$$
It has eigenvalue $3$ on the singlet and eigenvalue $-1$ on the triplet
subspace, whereas its product-state numerical radius is $1$. In fact, using Bloch coordinates of vectors $e_1, e_2$ we see that
$$
|\braket {H_2^{(0)} e_1\otimes e_2, e_1\otimes e_2}| = |a_1^x a_2^x +a_1^ya_2^y+ a_1^za_2^z|\le 1\,.
$$
Now let $E=e_1\otimes e_2\otimes\dots \otimes e_n$. Then
$$
|\braket{H_d^{(0)}E, E}| \le |\braket {\mathcal R_d E,E }= |\prod_{r=1}^d (\mathbf i a_r^x + \mathbf j a_r^y + \mathbf k a_r^z)| \le 1\,.
$$

\subsection{Unitary transformations as in $X, Y$ construction}

The above  coordinate expression of $\vec h_r:=(H_r^{(0)}, H_r^{(0)}, H_r^{(0)}, H_r^{(0)})$ via $\vec h_{r-1}=(H_{r-1}^{(0)}, H_{r-1}^{(0)}, H_{r-1}^{(0)}, H_{r-1}^{(0)})$ can be written in matrix form almost exactly as in \eqref{F} (but of course a bit more complicated).  In fact, let us see a full analogy. Relationship \eqref{F} can be written down in this form (we write only first $r$ qubit part, the rest is identity):
$$
\begin{pmatrix} Q_r, P_n\\
Q_r, P_r\end{pmatrix} = \begin{pmatrix} X, -Y\\
\!\!\!Y, \,\,\,X \end{pmatrix} \otimes \begin{pmatrix} Q_{r-1}, P_{r-1}\\
Q_{r-1}, P_{r-1}\end{pmatrix} \,.
$$
Analogously, the relations (5)-(8) is nothing else as
$$
\vec h_r = 
\begin{pmatrix}
0&-X&-Y&-Z\\
X&0&Z&-Y\\
Y&-Z&0&X\\
Z&Y&-X&0
\end{pmatrix}
\otimes \vec h_{r-1}\,.
$$

As in the $X, Y$ case this matrix becomes unitary if we replace $X, Y, Z$ matrices by numbers $x, y, z$ such that $x^2+y^2+z^2=1$. And this is another explanation why $\|H_d^{(0)}\|_{op} =1$.

\subsection{Counting scalar quaternion words}

Let $A_\ell$ be the number of words of length $\ell$ in
$\{\mathbf i,\mathbf j,\mathbf k\}$ whose product is scalar, and let $B_\ell$
be the number whose product is non-scalar.  Thus $A_0=1$ and $B_0=0$.

\begin{lemma}[Exact scalar-word count]
For every $\ell\geq0$,
\begin{equation}
A_\ell=\frac{3^\ell+3(-1)^\ell}{4}.
\label{eq:Aell}
\end{equation}
In particular, the number of nonzero Pauli coefficients in $H_{2m}^{(0)}$ is
\begin{equation}
N_m:=A_{2m}=\frac{9^m+3}{4}.
\label{eq:Nm}
\end{equation}
\end{lemma}

\begin{proof}
Appending any imaginary unit to a scalar word produces a non-scalar word.  If a
word is non-scalar, exactly one of the three possible appended units produces a
scalar product and the other two remain non-scalar.  Therefore
\begin{equation}
A_{\ell+1}=B_\ell,
\qquad
B_{\ell+1}=3A_\ell+2B_\ell.
\label{eq:countrec}
\end{equation}
The transition matrix has eigenvalues $3$ and $-1$, and the initial condition
gives \eqref{eq:Aell}.  Equation \eqref{eq:Nm} follows by setting
$\ell=2m$.
\end{proof}

\subsection{Two interesting mixed states}

For distinct sites $a,b$, define vector
\begin{equation}
\ket{s_{ab}}
:=
\frac{\ket{0_a1_b}-\ket{1_a0_b}}{\sqrt2}.
\label{eq:singlet}
\end{equation}
On a cycle of $2m$ qubits, consider two convex combinations of product states  (call them dimer coverings):
\begin{align}
\ket{D_0}
&:=
\ket{s_{1,2}}\ket{s_{3,4}}\cdots\ket{s_{2m-1,2m}},
\label{eq:D0}\\
\ket{\widetilde D_1}
&:=
\ket{s_{1,2m}}
\ket{s_{2,3}}\ket{s_{4,5}}\cdots\ket{s_{2m-2,2m-1}}.
\label{eq:D1tilde}
\end{align}
Their transition graph is a single loop.

\begin{lemma}[Dimer overlap]
The overlap is
\begin{equation}
\langle D_0,\widetilde D_1\rangle
=(-1)^{m-1}2^{1-m}.
\label{eq:raw-overlap}
\end{equation}
Hence, after setting
\begin{equation}
\ket{D_1}:=(-1)^{m-1}\ket{\widetilde D_1},
\qquad
s_m:=\langle D_0,D_1\rangle=2^{1-m},
\label{eq:positive-overlap}
\end{equation}
the overlap is positive.
\end{lemma}

\begin{proof}
A common computational-basis string occurs in both dimer very rarely, namely  only if  in the string neighboring
bits alternate around the cycle.  There are exactly two such strings,
$0101\cdots01$ and $1010\cdots10$.  Each contributes magnitude $2^{-m}$ to
the overlap, and the orientation convention in \eqref{eq:singlet} gives the
same sign $(-1)^{m-1}$ for both contributions.
\end{proof}

\subsection{Pauli coefficients}

For $\mathbf a=(a_1,\ldots, a_d)\in[3]^d$ and a family of quaternions, put 
\begin{equation}
c_{\mathbf a}
:=
\text{Scalar}(\ee_{a_1}\ee_{a_2}\cdots\ee_{a_d}).
\label{eq:cdef}
\end{equation}
Since a product of quaternion units belongs to
$\{\pm1,\pm\mathbf i,\pm\mathbf j,\pm\mathbf k\}$, one has
$c_{\mathbf a}\in\{0,\pm1\}$.  Expanding \eqref{eq:Rdef} gives
\begin{equation}
H_d^{(0)}
=
\sum_{\mathbf a\in[3]^d}
 c_{\mathbf a}\,
 \sigma_{a_1}\otimes\cdots\otimes\sigma_{a_d}.
\label{eq:Hcoeff}
\end{equation}
For even length $d=2m$, the standard matrix representation
\begin{equation}
\rho(\ee_a)=-i\sigma_a
\label{eq:rho}
\end{equation}
is an algebra homomorphism $\Hq\to M_2(\bC)$, and scalar part is half the trace.
Thus
\begin{equation}
c_{\mathbf a}
=
\frac{(-1)^m}{2}
\Tr(\sigma_{a_1}\sigma_{a_2}\cdots\sigma_{a_{2m}}).
\label{eq:ctrace}
\end{equation}
In particular, the coefficients of $H_{2m}^{(0)}$ are invariant under cyclic
rotation of the sites.

\bigskip  

\subsection{Computing $\langle D_0,P_{\mathbf a}D_1\rangle$}
The next identity is the key point: how to reproduce
the scalar quaternion coefficient.

\begin{lemma}
For every $\mathbf a=(a_1,\ldots,a_{2m})\in[3]^{2m}$, let
\[
P_{\mathbf a}:=\sigma_{a_1}\otimes\cdots\otimes\sigma_{a_{2m}}.
\]
Then
\begin{equation}
\langle D_0,P_{\mathbf a}D_1\rangle
= 2^{1-m}\,c_{\mathbf a}.
\label{eq:loopidentity}
\end{equation}
\end{lemma}

\begin{proof}
Let
\[
\varepsilon=
\begin{pmatrix}0&1\\-1&0\end{pmatrix}=iY = \ket 0\bra 1- \ket 1 \bra 0\,.
\]
so that
$\ket s=2^{-1/2}\sum_{u,v}\varepsilon_{uv}\ket{u,v}$.
This gives
\begin{align}
2^m\langle D_0,P_{\mathbf a}\widetilde D_1\rangle
=
\Tr\Big[&
(\sigma_{a_1}^{T}\varepsilon\sigma_{a_2})\varepsilon
(\sigma_{a_3}^{T}\varepsilon\sigma_{a_4})\varepsilon
\cdots
\notag\\[-1mm]
&\cdots
(\sigma_{a_{2m-1}}^{T}\varepsilon\sigma_{a_{2m}})
\varepsilon
\Big].
\label{eq:indexcontraction}
\end{align}
For each Pauli matrix,
\begin{equation}
\sigma_a^T\varepsilon=-\varepsilon\sigma_a,
\qquad
\varepsilon^T=-\varepsilon,
\qquad
\varepsilon^2=-I.
\label{eq:epsilonidentities}
\end{equation}
Substitution into \eqref{eq:indexcontraction}, followed by cyclicity of trace,
gives
\begin{equation}
\langle D_0,P_{\mathbf a}\widetilde D_1\rangle
=-2^{-m}\Tr(\sigma_{a_1}\cdots\sigma_{a_{2m}}).
\label{eq:rawloop}
\end{equation}
Multiplying by the phase $(-1)^{m-1}$ in \eqref{eq:positive-overlap} and using
\eqref{eq:ctrace},
\[
\langle D_0,P_{\mathbf a}D_1\rangle
=(-1)^m\, 2^{-m}\Tr(\sigma_{a_1}\cdots\sigma_{a_{2m}})
=2^{1-m}c_{\mathbf a}.
\]
This is \eqref{eq:loopidentity}.
\end{proof}

\subsection{The  lower bound}

For the remainder of the proof, abbreviate
\begin{equation}
H_{2m}:=H_{2m}^{(0)}.
\label{eq:Hshort}
\end{equation}

\begin{lemma}[Dimer matrix elements]
One has
\begin{align}
\langle D_0,H_{2m}D_0\rangle
&=3^m,
\label{eq:diag0}\\
\langle D_1,H_{2m}D_1\rangle
&=3^m,
\label{eq:diag1}\\
\langle D_0,H_{2m}D_1\rangle
&=2^{1-m} N_m,
\label{eq:offdiag}
\end{align}
where $N_m=(9^m+3)/4$.
\end{lemma}

\begin{proof}
For a single singlet,
\begin{equation}
\langle s,\sigma_a\otimes\sigma_b\,s\rangle=-\delta_{ab}.
\label{eq:singletcorr}
\end{equation}
Thus a Pauli word contributes to the $D_0$ diagonal matrix element only when
$a_{2r-1}=a_{2r}$ for every $r$.  There are $3^m$ such words.  For each one,
the quaternion coefficient is $(-1)^m$, while the product of singlet
correlations is also $(-1)^m$; hence every surviving contribution equals $1$.
This proves \eqref{eq:diag0}.

The coefficient formula \eqref{eq:ctrace} is cyclically invariant.  Therefore
$H_{2m}$ is invariant under cyclic translation of the sites.  The state $D_1$
is, up to a global phase, the one-site translate of $D_0$, so
\eqref{eq:diag1} follows.

Finally, by \eqref{eq:Hcoeff} and the loop identity,
\[
\langle D_0,H_{2m}D_1\rangle
=
\sum_{\mathbf a}c_{\mathbf a}
\langle D_0,P_{\mathbf a}D_1\rangle
=
2^{1-m}\sum_{\mathbf a}c_{\mathbf a}^2
=2^{1-m} N_m.
\]
\end{proof}

\begin{theorem}[Quaternionic $XYZ$ product-state gap]
For every $m\geq1$, the $2m$-qubit, $2m$-homogeneous Hermitian operator
$H_{2m}$ defined in \eqref{eq:Hshort} satisfies
\begin{equation}
\norm{H_{2m}}_{\ps}=1
\label{eq:thmps}
\end{equation}
and
\begin{equation}
\norm{H_{2m}}_{\op}
\geq
L_m
:=
\frac{9^m+2\cdot 6^m+3}{2^{m+1}+4}.
\label{eq:Lm}
\end{equation}
\end{theorem}

\begin{proof}
Equation \eqref{eq:thmps} has been shown before.  For the operator norm,
use the  combination dimer vector
\begin{equation}
\ket{\Psi_m}:=\ket{D_0}+\ket{D_1}.
\label{eq:Psi}
\end{equation}
By \eqref{eq:positive-overlap},
\[
\langle\Psi_m,\Psi_m\rangle=2(1+2^{1-m}).
\]
The dimer matrix elements give
\[
\langle\Psi_m,H_{2m}\Psi_m\rangle
=2(3^m+2^{1-m} N_m),
\]
where $N_m=(9^m+3)/4$.
This therefore yields
\begin{equation}
\norm{H_{2m}}_{\op}
\geq
\frac{3^m+2^{1-m}\,N_m}{1+2^{1-m}}.
\label{eq:rayleighintermediate}
\end{equation}
Substitute and $N_m=(9^m+3)/4$:
\[
\frac{3^m+2^{1-m}(9^m+3)/4}{1+2^{1-m}}
=
\frac{9^m+2\cdot6^m+3}{2^{m+1}+4}> (9/2)^m =(3/\sqrt{2})^d,
\]
if $2m=d$.
\end{proof}

$$
3/\sqrt{2} = 2.121320...
$$

\section{Non-commutative Bohnenblust--Hille inequality and the tightness of its constant}
\label{ncBH-s}

For hamiltonian $H=\sum_\alpha c_\alpha \sigma^\alpha$ of degree at most $d$, that is for those for which $c_\alpha=0$ as long as multi-index $\alpha \in \{0,1,2,3\}^n$
satisfies $|\text{supp}\alpha| >d$ the following estimate was  proved in \cite{VZ}:
\begin{equation}
\label{ncBH}
\Big(\sum_\alpha |c_\alpha|^{\frac{2d}{d+1}}\Big)^{\frac{d+1}{2d}} \le C(d) \|H\|_\op,
\end{equation}
where $C(d) \le 3^d$ is independent of the number of qubits $n$. In \cite{BSVZ} this estimate  for homogeneous hamiltonians of degree $d$ was improved to
\begin{equation}
\label{bsvz}
C(d) \le 3^{d/2} BH^d_{\pm} \le 3^{d/2+ c\sqrt{d\log d}} = 3^{\frac{d}2(1+o_d(1))}\,.
\end{equation}

Now we will repeat the proof of \eqref{bsvz} and show that it is practically tight by using \cite{S}.
\begin{remark}
Let us remind the reader that the order of magnitude of BH constant $BH^d_{\pm}$ is still unknown and its separation between
$BH^d_{\pm}\le e^{c\sqrt{d\log d}} $  and $\le d^C$ will play an important part in partially clarifying Aaronson--Ambainis conjecture.
In \cite{S} it is shown that non-commutative constant $BH^d_{M(\bC^2)}$ {\bf must be exponential} in $d$. We borrow \cite{S} proof and write it down below.
\end{remark}

First let us prove \eqref{bsvz}. To this end consider scenarios, where scenario is just a map
$$
s: [n]\to \{1,2,3\}^n\,. 
$$
Consider the part of homogeneous $H$ controlled by scenario:
$$
H_s=\sum_{\alpha \le s} c_\alpha\, \sigma^\alpha,
$$
where $\alpha\le s$ means that for every $j\in [n]$ $\alpha_j =s(j)$ or $\alpha_j=0$.

Now let us notice that every $H_s$ can be made a Hamiltonian in which only $X=\sigma^{(1)}= X$ Pauli's are involved just by a local unitary, that is by 
conjugating with $U_1\otimes \dots\otimes U_n$ (all $U_j$ are unitaries in corresponding qubit $\bC^2$).
But this mean that every $H_s$ is commutative, and 
\begin{equation}
\label{LP}
\|H_s\|_\op =\|h_s\|_\infty,
\end{equation}
where $h_s(x)$ is obtained by replacing each $X_j$ by $x_j \in \{\pm 1\}$. This is clear by the fact of simultaneous diagonalization of all monomials,
see \cite{BLP}.

Hence
$$
\Big(\sum_{\alpha\le s} |c_\alpha|^{\frac{2d}{d+1}}\Big)^{\frac{d+1}{2d}} \le BH^d_{\pm} \|H_s\|_\op \le e^{c\sqrt{d\log d}}\|H_s\|_\op\,.
$$
Now we will add these inequalities keeping in mind that each $c_\alpha$ is met in exactly $3^{n-d}$ scenarios $s$.

\bigskip

Let $e^{\alpha}_x$, $\alpha \in \{1,2,3\}, x\in \{-1,1\}$ denote the eigenvector of Pauli $\sigma^\alpha$ with eigenvalue $x$.
The useful formula is
$$
\ketbra{e^{\alpha}_{x}}{e^{\alpha}_{x}} =\frac12 \sigma^{(0)} + \frac12 \, x\, \sigma^\alpha\,.
$$
First we write a formula for $H_s$: let $x\in \{-1,1\}^n$, put
\begin{equation}\label{r-s}
\rho_{x, s}
:=\ketbra{e^{s(1)}_{x_1}}{e^{s(1)}_{x_1}}\otimes\dots\otimes  \ketbra {e^{s(n)}_{x_n}}{e^{s(n)}_{x_n}} 
=\Big(\frac12\sigma_0+ \frac12 x_1 \sigma_{s(1)}\Big)\otimes\dots\otimes  \Big(\frac12\sigma_0+ \frac12 x_n \sigma_{s(n)}\Big),
\end{equation}

For any $s : [n]\to \in[3]^n$, we have the formula for $H_s$:
\begin{equation}\label{defn:es}
H_{s}=\mathcal E_s (H):=\sum_{x\in\{-1,1\}^n}\rho_{x,s}H\rho_{x,s}.
\end{equation}

\bigskip

The operator $\mathcal E_s (H)$ is  completely positive, $\mathcal E_s^2= \mathcal E_s$ and it is unital. Therefore it is
the conditional expectation onto the commutative sub-algebra $\mathcal{H}_{s}$ generated by 
$$
\un\otimes \cdots \otimes \sigma_{s(j)}\otimes  \cdots \otimes \un,\qquad j\in [n]
$$
where $\sigma_{s(j)}$ appears in the $j$-th place, see \cite{DR}, \cite{T}.
It also is related to the $n$-fold tensor product of the $1$-qubit depolarizing channel with parameter $1/3$ employed by \cite{BGKT} via averaging over $s$'s.

The only thing we need is a slightly non-trivial inequality
\begin{equation}
\label{norm1}
\|H_s\|_\op \le \|H\|_\op\,.
\end{equation}

\begin{align*}
&\sum_\alpha |c_\alpha|^{\frac{2d}{d+1}}= 3^{d-n} \sum_s \sum_{\alpha\le s} |c_\alpha|^{\frac{2d}{d+1}}\le
\\
& 3^{d-n} \sum_s \Big(BH^d_{\pm} \, \|H_s\|_\op\Big)^{\frac{2d}{d+1}} \le 3^{-n} \sum_s \Big(3^{\frac{d+1}{2}}BH^d_{\pm} \, \|H\|_\op\Big)^{\frac{2d}{d+1}},
\end{align*}
which immediately gives the estimate \eqref{bsvz}: $C(d) \le 3^{\frac12(d+o_d(1))}$.

\bigskip

We already noticed a non-trivial inequality \eqref{norm1}. Let us give a simple proof.

For that consider operator $D$ on matrices that take a $2^n\times 2^n$ matrix and sends to its diagonal part.

$$
H\to D(H)\,.
$$
Obviously it is bounded in operator norm. It is also clear that any monomial $\sigma^\alpha$ having at least on $\sigma^{(1)}=X$ or $\sigma^{(2)}=Y$ will satisfy $D(\sigma^\alpha)=0$.
Only monomials having only $\sigma^{(3)}=Z$ survive.

So, restricting to diagonal send any Hamiltonian $H$ to its scenario $H_{t}$, where $t:[n] \to \{3\}$.

\medskip

For example, let
$$
U=\frac1{\sqrt{2} }\begin{pmatrix}
&1, &1\\
&-1, &1
\end{pmatrix},
$$
and 
$$
C(H) = U\otimes \dots \otimes U H U^*\otimes \dots \otimes U*,
$$
Then $C^{-1}D C$ will map any Hamiltonian $H$ into $H_s$, where $s:[n]\to \{1\}$.

All transformations here are bounded in operator norm. Any other scenarios hamiltonian $H_s$ can be obtained from $H_t$ by the conjugation with corresponding 
$U_1\otimes \dots \otimes U_n$, where $U_j$  is $2\times 2$ unitary on $j$-th qubit.

\subsection{The tightness of exponential estimate for non-commutative BH}
We repeat the construction of \cite{S}.
Let $k+r=d$ and we split $d$ is approximately in half (or exactly if $d$ is even).
First we construct $3^r$ Pauli strings $\{P_1,\dots, P_r\}$ on $n=\frac{3^r-1}{2} $ qubits that have properties:

1) strings $P_1,\dots, P_r$ are pairwise anti-commuting;

2) the degree (weight) of each string is $r$.

\medskip

This was achieved in \cite{JKMN}. Here is a simple explanation of their construction. For $r=1$ we put $X, Y, Z$ in a column, one qubit each.
Now for $r=2$ we have $4$ qubits and should have $9$ strings:
complement the $X, Y, Z$ column by only $I$ in the next $2$ qubits, we have now $3\times 3$ matrix. Under it put another $3\times 3$ matrix and
fill it in by $I $'s in the first qubit, $X, Y, Z$ column in the second qubit and $I$'s in the third qubit.
Now put under the obtained $6\times 3$ matrix another $3\times 3$ matrix and fill its first qubits by $I$'s and fill the third qubit by $X, Y, Z$ column.

\medskip

After this we have $9\times 3$ matrix and we still have only $3$ qubits used. To fill the fourth qubit we put $X, X, X$ column in the first three rows in the fourth qubit and then $Y, Y, Y$ column in the next three rows in the fourth qubit, and finally, put column $Z, Z, Z$ in the last three rows in the fourth qubit. The construction of $3^2$ Pauli rows on $4=(3^2-1)/2$ qubits is finished.

One easily checks that they all anti-commute and the weight (degree) of each is $r=2$.

We call this $9\times 4$ matrix  the essential block $E_2$ for $r=2$. The essential block $E_1$ for $r=1$ was just column $X, Y, Z$.

\medskip

Now we just repeat the construction for $r=3$ in a self-similar fashion. We
put $E_2$ in the left upper corner, and next to it two $9\times 4$ blocks filled by $I$'s, next we
put another $9\times 4$ block filled by $I$'s just below $E_2$, next to it we put a copy of $E_2$, and next to it another $9\times 4$ block filled by $I$'s.
Below all that we put another $9\times 4$ block filled by $I$'s and yet another $9\times 4$ block filled by $I$'s, and then in the qubits nine to twelve a copy of $E_2$. We filled $27$ rows and $12$ qubits. For $r=3$ the number of qubits is $n=(3^3-1)/2=13$. So we have one extra last qubit.

\medskip

We put in it the column of $X$'s of hight $9$ in the first nine rows, in the next nine rows we put a column of $Y$'s of height $9$, and in the last nine rows in the thirteenth qubit we put the column of $Z$'s.

We filled in the matrix $27\times 13$, call it essential block $E_3$ for $r=3$. It is obvious to check that all rows anti-commute.

\medskip

To build $81\times 40$ essential block $E_4$ we just repeat this in a self similar fashion using $E_3$ as we just used $E_2$. Et cetera...

\bigskip

So these were $P$ strings. Consider now a totally different collection of Pauli string. We will call them $B$ strings. Let $L=2^k$ be the largest smaller than $3^r$.

Let $L$ be the largest power of two below $3^r$.  

\begin{equation}
\label{Lm}
L= 2^m\,.
\end{equation}

Choose $L$ elements of the family in of $P_j$ above and denote them
\[
 P_1,\ldots,P_L.
\]
Thus
\[
 P_j^2=\Id,
 \qquad
 P_j P_i=-P_iP_j\quad(j\ne i),
 \qquad
 \wt(P_j)=r.
\]

Let $H\in\{\pm1\}^{L\times L}$ be a Sylvester Hadamard matrix:
\[
 H^\mathsf TH=L I_L.
\]
The auxiliary register has $k$ blocks of $L$ qubits.  For $s\in[k]$ and $a\in[L]$, let $Z_{s,a}$ be Pauli $Z$ on coordinate $a$ in block $s$.  Introduce the operator-valued diagonal matrices
\[
 \mathsf D_s=\operatorname{diag}(Z_{s,1},\ldots ,Z_{s,L}).
\]
All their entries commute.  Define the column vector of $B$-operators by
\begin{equation}\label{eq:B-vector}
 \boxed{
 \begin{pmatrix}B_1\\ \vdots\\ B_L\end{pmatrix}
 =H\mathsf D_1H\mathsf D_2\cdots H\mathsf D_k\un.
 }
\end{equation}
For $k=0$, set $B_\ell=\Id$.

Matrix $H\mathsf D_1H\mathsf D_2\cdots H\mathsf D_k$ is $L2^{kL}\times L2^{kL}$, and $\un$ is a $L2^{kL}\times 2^{kL}$ tall matrix consisting of 
$2^{kL}\times 2^{kL}$ identity matrices.

\subsection{The sum-of-squares identity}
\label{sumOFsquares}

\begin{lemma}[Hadamard suffix identity]\label{lem:B-square}
The operators in \eqref{eq:B-vector} satisfy
\[
 \sum_{\ell=1}^LB_\ell^2=L^{k+1}\Id.
\]
\end{lemma}

\begin{proof}
Write $\mathbf B=(B_1,\ldots , B_L)^\mathsf T$.  Since all operator entries commute and each $\mathsf D_s$ is a self-adjoint involution,
\begin{align*}
 \sum_{\ell=1}^LB_\ell^2
 &=\mathbf B^\mathsf T\mathbf B\\
 &=\un^\mathsf T\mathsf D_kH^\mathsf T\cdots
   \mathsf D_1H^\mathsf TH\mathsf D_1\cdots H\mathsf D_k\un\\
 &=L\,\un^\mathsf T\mathsf D_kH^\mathsf T\cdots
   \mathsf D_2H^\mathsf TH\mathsf D_2\cdots H\mathsf D_k\un\\
 &=\cdots=L^k\un^\mathsf T\un\,\Id=L^{k+1}\Id.
\end{align*}
At each step we use $H^\mathsf TH=LI$ and $\mathsf D_s^2=I$.
\end{proof}

Set
\[
 H_d=\sum_{\ell=1}^LP_\ell\otimes B_\ell,
 \qquad
 N_d=L^{k+1},
 \qquad
 O_d=N_d^{-1/2}H_d.
\]
The total number of qubits is
\[
 n_d=\frac{3^r-1}{2}+kL.
\]

\begin{theorem}[Flat homogeneous quantum Boolean function]\label{thm:flat}
The operator $O_d$ is a balanced quantum Boolean function.  It has exactly $N_d$ nonzero Pauli coefficients, every one of magnitude $N_d^{-1/2}$, and every supported Pauli string has weight exactly $d$.
\end{theorem}

\begin{proof}
The cross terms in $H_d^2$ cancel:
\begin{align*}
 H_d^2
 &=\sum_\ell \Id\otimes B_\ell^2
   +\sum_{\ell<m}(P_\ell P_m+P_mP_\ell)\otimes B_\ell B_m\\
 &=\Id\otimes\sum_\ell B_\ell^2=N_d\Id,
\end{align*}
where we used anti-commutation of the $P_\ell$ and commutation of the $B_\ell$.  Hence $O_d^2=\Id$.

\medskip

Let us write the formula for $B_\ell$. We can think that index $\ell\in [L]=[2^k]$ is written as an element of $\{0,1\}^k$.
\begin{equation}
\label{eq:B-expanded}
B_\ell =\sum_{a_1,\dots, a_k\in \{0,1\}^m} (-1)^{\ell\cdot a_1}(-1)^{a_1\cdot a_2}\dots (-1)^{a_{k-1}\cdot a_k} Z_{1, a_1}\dots Z_{k, a_k},
\end{equation}
where $m$ is from \eqref{Lm}.

\medskip

For each $\ell$, formula \eqref{eq:B-expanded} contains $L^k$ distinct  strings.  Different values of $\ell$ have different  strings $P_\ell$, so there are no collisions between the corresponding full Pauli strings.  Thus $H_d$ has $L^{k+1}=N_d$ coefficients of magnitude one.  Normalization gives magnitude $N_d^{-1/2}$.  Every term has  weight $r$ in $P$ part and weight $k$ in $B$ part, hence total weight $d=r+k$.  Since $d\ge1$, there is no identity term.
\end{proof}

Now $H_d=\sum_{\ell=1}^L P_\ell\otimes B_\ell$ has norm $L^{\frac{k+1}{2}}\asymp 3^{r\frac{k+1}{2}}= 3^{d^2/8}$.

\medskip

On the other hand all strings $P_\ell\otimes B_\ell$ are have different supports (because any two $P$-strings have different supports)  and there are exactly $ L^{k+1}$ of them. All coefficients of $H_d$ are $\pm 1$.
Therefore measuring $\|\hat H_d\|_{\ell^p}$ returns $L^{\frac{k+1}{p}}$. 

Therefore the ratio
$$
\frac{\|\hat H\|_{\ell^p}}{\|H\|_\op} = L^{(k+1)(1/p-1/2)} = L^{\frac{k+1}{2d}} \ge 3^{d/8},
$$
if $p=\frac{2d}{d+1}$. Non-commutative BH constant must be exponential in degree.

\section{Discussion of $P-B$ construction above}
\label{disc}

There is a lingering feeling that the previous constructions can be useful for other quantum estimates, mainly for disproving the quantum analogs of classical commutative results.

One problem is the already mentioned {\bf Aaronson--Ambainis conjecture}.
In commutative world (where it was stated) it looks like that: given $f:\{-1,1\}^n\to \bR$,  $\|f\|_\infty \le 1$, $\text{deg} f\le d$, is it true that there are two absolute constants $C, K$ such that
\begin{equation}
\label{AA}
\sum_{S\subset [n], S\neq \emptyset} \hat f(S)^2 =:\Var[f]  \le C d^K \text{maxInf}[f]?
\end{equation}

\begin{remark}
Sometimes it is written slightly differently by raising $\Var$ also to power $K$, but this is not essential.
\end{remark}

Here $\text{maxInf}[f]=\max_{j=1}^n \Inf_j[f]$ and 
$$
\Inf_j[f] =\bE|D_j f|^2 =\sum_{S\subset [n], j\in S} \hat f(S)^2
$$ 
is the influence of $j$-th coordinate.  Aaronson--Ambainis conjecture claims the existence (in a very precise sense) of an influential variable for any polynomial of degree $d$ on Hamming cube that is globally bounded by $1$.

\medskip 
This is still open, and the quantum version states the following:
given traceless hamiltonian $H$,  $\|H\|_\op \le 1$, $\text{deg} \,H\le d$, is it true that there are two absolute constants $C, K$ such that
\begin{equation}
\label{AAQ}
\Var[H] := \| H\|_{HS}^2 \le C d^K \text{maxInf}[f]?
\end{equation}
Here $\text{maxInf}[H]=\max_{j=1}^n \Inf_j[H]$ and 
$$
\Inf_j[H] =\|D_j H\|_{HS}^2\,.
$$ 
The Hilbert--Schmidt norm (HS) is normalized, namely, as we are in $\bC^{2^n}$ it is
$$
\|A\|_{HS}^2 := 2^{-n} \tr [A^*A]\,.
$$
In particular, unusually
$$
\|A\|_{HS} \le \|A\|_{\op}\,.
$$
Classical {\bf Aaronson--Ambainis conjecture}  \eqref{AA} is open. There is a hope to show that its quantum version \eqref{AAQ} can have a more or less easy counterexample.
 However, below we prove quantum Aaronson--Ambainis conjecture for a class of very non-commutative boolean functions.
 
 In the scalar case the Aaronson--Ambainis conjecture is proved for all boolean functions, see \cite{OSSS}, \cite{Z}.

\bigskip

\section{Disproving quantum Entropy-Influence conjecture}
\label{EICdis}

Another commutative problem on Hamming cube which got a lot of attention, but is still open is {\bf Entropy-Influence conjecture (EIC)}.

Again let $f=\sum_{S\subset [n]} \hat f(S) \xi_S$ be a polynomial of degree at most $d$ on Hamming cube. But now we also assume that it is boolean, $f:\{-1,1\}^n\to \{-1,1\}$.  Then EIC requires to prove that there exists an absolute constant such that
\begin{equation}
\label{EIC}
\text{Ent}[\hat f]:=\sum_{S\subset [n], S\neq \emptyset} \hat f(S)^2\cdot \log\frac{1}{\hat f(S)^2} \le C\, \Inf[f],
\end{equation}
where $\Inf[f] :=\sum_{j=1}^n \Inf_j[f] =\sum_{S\subset [n]} |S|\hat f(S)^2$.

\bigskip

Quantum boolean functions (hamiltonians) are such $H$ that
$$
\|H\|_{HS} =\|H\|_\op\,.
$$
There are many interesting things about them, \cite{MO}. 

Next we disprove the quantum version of \eqref{EIC}.


\subsection{Disproving quantum Entropy-Influence conjecture \eqref{EIC}}
Bu, Garcia, Jaffe, Koh, and Li  \cite{BuEtAl} conjectured that a universal constant $C$ should satisfy
\begin{equation}\label{eq:qfei}
 H(O)\le C I(O)
\end{equation}
for every qubit quantum Boolean function $O$ \cite[Conjecture 24]{BuEtAl}.

Let us normalize the hamiltonian $H_d$ constructed using $P-B$ method above. Consider
$$
O_d =\frac{1}{L^{\frac{k+1}{2}}} H_d\,.
$$
Notice that 
$$
\|O_d\|_{HS} =\|O_d\|_\op
$$
Then all Fourier coefficients of $O_d$ are $\pm \frac{1}{L^{\frac{k+1}{2}}} $ and there are exactly $L^{\frac{k+1}{2}} $ of them. 
Therefore, 
\begin{equation}
\label{EntO}
\text{Ent}[O_d] = \log L^{\frac{k+1}{2}} \asymp d^2\,.
\end{equation}

\medskip

But for any boolean function (quantum or classical)
$$
\Inf [F]  \le \text{deg}(F) \|F\|_{HS}^2\,.
$$
So for $O_d$ we have 
\begin{equation}
\label{InfO}
\Inf[O_d] \le d\,.
\end{equation}

Therefore, comparing \eqref{EntO} and \eqref{InfO}, we see that  there is no absolute constant $C$ such that
$$
\text{Ent}[H] \le C\, \Inf[H]
$$
for quantum boolean hamiltonian $H$.
 
 \subsection{Another counterexample to quantum Entropy-Influence conjecture \\ using weighted Rudin--Shapiro tensor construction}
 
 In \cite{GS} Gideon Schechtman gave an example of a complex function on boolean cube with values in the unit circle such that Entropy-Influence conjecture fails.
 We first indicate his approach.
 
 \subsection{A weighted scalar Rudin--Shapiro phase family}
 \label{wRS}

We now turn to a different, commutative realization.  Fix weights
$a_1,\dots,a_n\in(0,1]$, set $\mathsf P_0=\mathsf Q_0=1$, and define real
functions on $\{-1,1\}^k$ by
\begin{align}
\mathsf P_k&=\mathsf P_{k-1}+a_kx_k\mathsf Q_{k-1},\label{eq:scalar-P}\\
\mathsf Q_k&=-a_kx_k\mathsf P_{k-1}+ \mathsf Q_{k-1}.
\label{eq:scalar-Q}
\end{align}
A direct calculation gives
\begin{equation}
\mathsf P_k^2+\mathsf Q_k^2=(1+a_k^2)
(\mathsf P_{k-1}^2+\mathsf Q_{k-1}^2),
\end{equation}
so
\begin{equation}
\mathsf P_n(x)^2+\mathsf Q_n(x)^2
=2\prod_{j=1}^n(1+a_j^2).
\label{eq:scalar-complement}
\end{equation}
Define
\[
D_n=\sqrt{2\prod_{j=1}^n(1+a_j^2)},\qquad
f_n(x)=\frac{\mathsf P_n(x)+i\mathsf Q_n(x)}{D_n}.
\]
Then $|f_n(x)|=1$ for every $x$.

For $A\subseteq[n]$, put $a_A=\prod_{j\in A}a_j$.  There are signs
$\alpha_A,\beta_A\in\{\pm1\}$ such that
\[
\widehat{\mathsf P_n}(A)=\alpha_Aa_A,\qquad
\widehat{\mathsf Q_n}(A)=\beta_Aa_A,
\]
and therefore
\[
\widehat f_n(A)=\frac{a_A}{D_n}(\alpha_A+i\beta_A).
\]
With
\[
p_j=\frac{a_j^2}{1+a_j^2},
\]
the squared Fourier coefficients form the product measure
\begin{equation}
|\widehat f_n(A)|^2=
\prod_{j\in A}p_j\prod_{j\notin A}(1-p_j).
\label{eq:product-law}
\end{equation}
Consequently,
\begin{equation}
\Inf[f_n]=\sum_{j=1}^n p_j,
\qquad
\Hf[f_n]=\sum_{j=1}^n\hbin(p_j).
\label{eq:scalar-HI}
\end{equation}
Writing $a_j=\tan\theta_j$ gives $p_j=\sin^2\theta_j$; the angles are
classical coefficient parameters before they acquire a Pauli interpretation.

\subsection{Towards the quantum case: the Hermitian dilation mechanism}

The following construction works for every unimodular complex function, not
only the family above.  For
\[
f(x)=\sum_{A\subseteq[n]}\widehat f(A)\chi_A(x),
\]
replace $\chi_A$ by the commuting Pauli string
$X_A=\bigotimes_{j=1}^nX_j^{\mathbf1_{j\in A}}$ and define
\begin{equation}
F_f=\sum_A\widehat f(A)X_A=f(X_1,\dots,X_n).
\label{eq:Ff}
\end{equation}
In the common $X$-eigenbasis, $F_f\ket x_X=f(x)\ket x_X$.  Hence
\[
|f(x)|=1\ \forall x\quad\Longleftrightarrow\quad F_f\text{ is unitary}.
\]

\medskip
Write $F_f=A+iB$ with $A,B$ Hermitian,
add one ancilla qubit and set:

\begin{equation}
O_f=X_0\otimes\Re F_f-Y_0\otimes\Im F_f
=\sum_A\bigl(
\Re\widehat f(A)X_0X_A-
\Im\widehat f(A)Y_0X_A
\bigr).
\label{eq:pauli-dilation}
\end{equation}

\bigskip 

This operator can be written in matrix form (ancilla qubit doubles the dimension):
\begin{equation}
O_f=\begin{pmatrix}0&F_f\\F_f^\dagger&0\end{pmatrix}.
\label{eq:dilation}
\end{equation}
Then $O_f^\dagger=O_f$ and
\[
O_f^2=\begin{pmatrix}F_fF_f^\dagger&0\\0&F_f^\dagger F_f\end{pmatrix}=I.
\]
Thus $O_f$ is a Hermitian quantum Boolean function.

For the weighted Rudin--Shapiro phase,
\begin{equation}
O_n=\frac{X_0\otimes\mathsf P_n(X)-Y_0\otimes\mathsf Q_n(X)}{D_n}.
\label{eq:On}
\end{equation}
The data operators commute because they are polynomials in the commuting family
$X_1,\dots,X_n$.  Therefore \eqref{eq:scalar-complement} also gives a direct
Pauli verification that $O_n^2=I$.

\subsection{Exact transfer of entropy and influence}

For normalized Pauli coefficients, define
\[
\Hf[O]=-
\sum_\sigma|\widehat O(\sigma)|^2\log_2|\widehat O(\sigma)|^2,
\quad
\Inf[O]=\sum_\sigma\wt(\sigma)|\widehat O(\sigma)|^2,
\]
and
\[
\Inf_j^{(2)}(O)=\sum_{\sigma:\sigma_j\ne I}|\widehat O(\sigma)|^2.
\]

\begin{proposition}[Hermitian-dilation transfer formulas]
Let $|f|=1$ and let $O_f$ be given by \eqref{eq:dilation}.  Then
\begin{align}
\Inf_0^{(2)}(O_f)&=1,\label{eq:anc-inf}\\
\Inf_j^{(2)}(O_f)&=\Inf_j(f),\qquad 1\le j\le n,\label{eq:data-inf}\\
\Inf[O_f]&=1+\If[f].\label{eq:total-inf}
\end{align}
If $\widehat f(A)=|\widehat f(A)|e^{i\phi_A}$, then
\begin{equation}
\Hf[O_f]=\Hf[f]+
\sum_A|\widehat f(A)|^2\hbin(\cos^2\phi_A),
\label{eq:entropy-transfer}
\end{equation}
and hence
\begin{equation}
\Hf[f]\le\Hf[O_f]\le\Hf[f]+1.
\label{eq:entropy-bounds}
\end{equation}
\end{proposition}

\begin{proof}
Every Pauli term in \eqref{eq:pauli-dilation} is nonidentity on the ancilla,
which proves \eqref{eq:anc-inf}.  For a data coordinate $j$, the two pieces
coming from $A$ contribute exactly when $j\in A$, and their squared coefficients
sum to $|\widehat f(A)|^2$.  This proves \eqref{eq:data-inf} and
\eqref{eq:total-inf}.  Entropy is obtained by splitting the mass
$|\widehat f(A)|^2$ into its real and imaginary squares, giving
\eqref{eq:entropy-transfer}.
\end{proof}

For the Rudin--Shapiro phase family, the real and imaginary parts of every
Fourier coefficient has equal magnitude.  Hence the entropy gain is exactly one
bit:
\begin{equation}
\Hf[O_n]=1+\sum_{j=1}^n\hbin(p_j),\qquad
\Inf[O_n]=1+\sum_{j=1}^np_j,
\label{eq:exact-quantum-HI}
\end{equation}
and
\[
\Inf_0^{(2)}(O_n)=1,\qquad
\Inf_j^{(2)}(O_n)=p_j\quad(1\le j\le n).
\]

\subsection{Application to quantum Fourier Entropy--Influence}

The dimension-free QFEI statement asks for a universal $C$ such that every
Hermitian involution $O$ satisfies
\[
\Hf[O]\le C\Inf[O].
\]
Choose
\[
a_1=\cdots=a_n=\frac1{\sqrt n},\qquad
p_1=\cdots=p_n=\frac1{n+1}.
\]
Then
\begin{equation}
\Inf[O_n]=1+\frac n{n+1}=2-\frac1{n+1}<2,
\end{equation}
while
\begin{equation}
\Hf[O_n]=1+n\hbin\!\left(\frac1{n+1}\right)
=\log_2n+O(1).
\end{equation}
Thus $\Hf[O_n]/\Inf[O_n]\to\infty$.  Under the definitions above, the family
is an obstruction to a dimension-free QFEI inequality in this generality.
The logarithmic dimensional loss in the known weak bound is of the correct order
for this example.

\subsection{Relation between tensor Rudin--Shapiro construction in Section \ref{RSp} \\ and the construction in Section \ref{wRS}}

The constructions share a two-component orthogonal geometry but differ in where
it is realized.

In both cases, two real components are packaged as one complex quantity.  In the
tensor recursion, this yields the factor $X-iY=2\proj10$ and therefore an exact
GHZ transition.  In the scalar recursion, it yields a phase-valued function;
the Hermitian dilation then realizes each complex coefficient in the Pauli plane
spanned by $X_0$ and $Y_0$.

On the GHZ subspace, $2^{-n}Q_n$ and $2^{-n}P_n$ form an encoded Pauli pair.
Thus the first construction can be viewed as a logical-qubit realization of the
same $X/Y$ geometry used by the physical ancilla in the Hermitian dilation.  The
important distinction is that the tensor pair acts only on a two-dimensional
code sector, whereas the dilation is a full involution on the entire Hilbert
space.

\section{Remarks on quantum KKL and Aaronson--Ambainis conjecture}

The dilation family is not extremal for quantum KKL because
\[
\Inf_0^{(2)}(O_n)=1.
\]
The ancilla is maximally influential.  On the data coordinates,
\[
\max_{1\le j\le n}\Inf_j^{(2)}(O_n)=\frac1{n+1},
\]
but standard quantum KKL counts all qubits, including the ancilla.  Thus the
family is useful for normalizations and for showing why Hermitian dilation does
not transfer a unimodular classical KKL counterexample into a qubit KKL
counterexample.

The original Aaronson--Ambainis conjecture concerns bounded real scalar
low-degree polynomials.  The present operators are not of interesting instances.  As a
Pauli/operator calibration family, they again satisfy an influential-coordinate
conclusion trivially because of the ancilla.  Nevertheless, their Pauli weight
distribution is explicit: for $a_j=1/\sqrt n$ it is
\[
1+K_n,\qquad K_n\sim\operatorname{Binomial}\left(n,\frac1{n+1}\right).
\]
Thus the exact degree is $n+1$, while the family is approximable in $L^2$ by a
degree bound independent of $n$ at any fixed error.  This separates approximate
low degree from small Pauli Fourier entropy.

\section{Summary of sharp formulas}

For the tensor Pauli recursion of Section \ref{RSp} (non-boolean output):
\[
\boxed{
\deg P_n=\deg Q_n=n+1,
\quad
\|P_n\|_{\op}=\|Q_n\|_{\op}=2^n,
\quad
\|P_n\|_{\HS}=\|Q_n\|_{\HS}=2^{n/2},
\quad
\|P_n\|_{\ps}=\|Q_n\|_{\ps}=1.
}
\]
For a unimodular scalar function $f$ and its Hermitian dilation (boolean output),
\[
\boxed{
O_f=\begin{pmatrix}0&F_f\\F_f^\dagger&0\end{pmatrix},
\qquad
O_f^2=I,
\qquad
\If[O_f]=1+\If[f],
\qquad
\Hf[f]\le\Hf[O_f]\le\Hf[f]+1.
}
\]
For the weighted Rudin--Shapiro choice $a_j=1/\sqrt n$,
\[
\boxed{
\If[O_n]<2,
\qquad
\Hf[O_n]=\log_2n+O(1).
}
\]

\subsection{Back to quantum Aaronson--Ambainis conjecture. Can we diminish the maximal qubit load of anti-commuting Pauli strings?}
\label{max}

One would wish to disprove the quantum Aaronson--Ambainis conjecture by the same $O_d$ hamiltonian.

Unfortunately, when calculating $\text{maxInf }[H_d]$ (or its normalized version $O_d$) we will meet a very special qubit that is loaded by {\bf all} strings $P_\ell\otimes B_\ell$.
It is this special qubit having the property that all strings of $H_d$ have a non-trivial Pauli over it, the so called {\bf root qubit}. If the reader remembers it is the fourth qubit in $E_2$ and the thirteenth qubit in $E_3$ and $40$-th qubit in $E_4$... .

\bigskip 

Here is a natural question.
Can one construct $3^r$ Pauli strings which are pairwise anti-commuting, such that
\begin{enumerate}
    \item each qubit is nontrivial in at most $3^{r/2}$ strings, and
    \item each string has weight at most $10r$,
\end{enumerate}
with no restriction on the total number of qubits? In other words can one {\bf spread} non-trivial Pauli's in $P$-strings to make the load lighter over each
qubit and keeping the weight linear in $r$ and having exponentially in $r$ many anti-commuting strings?

\bigskip

No, not for general $r$. In fact, the three requirements are inconsistent for every integer $r \ge 7$.

\begin{theorem}
Let $\mathcal{S}$ be a set of $N$ pairwise anti-commuting Pauli strings. Suppose every string has weight at most $w$, and every qubit is nontrivial in at most $L$ strings. Then
\[
    w \ge \frac{3(N-1)}{2L}.
\]
\end{theorem}

\begin{proof}
For a qubit $q$, let $x_q,y_q,z_q$ be the numbers of strings in $\mathcal{S}$ that have $X,Y,Z$ respectively on that qubit. Let
\[
    d_q=x_q+y_q+z_q.
\]
By assumption, $d_q\le L$ for every qubit $q$.

For two strings to locally anti-commute on qubit $q$, their non-identity Paulis on $q$ must be different. Hence the number of unordered pairs of strings that locally anti-commute on qubit $q$ is
\[
    x_qy_q+x_qz_q+y_qz_q.
\]
Using $d_q=x_q+y_q+z_q$, we have
\[
    x_qy_q+x_qz_q+y_qz_q
    =\frac{d_q^2-x_q^2-y_q^2-z_q^2}{2}
    \le \frac{d_q^2}{3}
    \le \frac{Ld_q}{3}.
\]
Now sum this upper bound over all qubits. The total number of local anti-commutation events is at most
\[
    \sum_q \frac{Ld_q}{3}
    =\frac{L}{3}\sum_q d_q.
\]
The quantity $\sum_q d_q$ is the total weight over all strings in $\mathcal{S}$, so
\[
    \sum_q d_q \le Nw.
\]
Therefore the total number of local anti-commutation events is at most
\[
    \frac{LNw}{3}.
\]

On the other hand, since the $N$ Pauli strings are pairwise anti-commuting, every unordered pair of distinct strings must locally anti-commute on at least one qubit. Thus the number of local anti-commutation events is at least
\[
    \binom{N}{2}.
\]
Combining the lower and upper bounds gives
\[
    \binom{N}{2}\le \frac{LNw}{3}.
\]
Canceling $N$ yields
\[
    \frac{N-1}{2}\le \frac{Lw}{3},
\]
which is equivalent to
\[
    w \ge \frac{3(N-1)}{2L}.
\]
\end{proof}

In the proposed parameters,
\[
    N=3^r,\qquad L=3^{r/2},\qquad w\le 10r.
\]
The theorem gives the necessary condition
\[
    w \ge \frac{3(3^r-1)}{2\cdot 3^{r/2}}
    =\frac{3}{2}\left(3^{r/2}-3^{-r/2}\right).
\]
Therefore any such construction must satisfy
\[
    10r \ge \frac{3}{2}\left(3^{r/2}-3^{-r/2}\right).
\]
But for $r=7$,
\[
    \frac{3}{2}\left(3^{7/2}-3^{-7/2}\right) \approx 70.116 > 70 = 10r.
\]
The left-hand side grows only linearly in $r$, while the required lower bound grows like $3^{r/2}$. Hence the inequality fails for every integer $r\ge 7$.

\begin{corollary}
There is no construction of $3^r$ pairwise anti-commuting Pauli strings satisfying both
\[
    \text{maximum qubit load} \le 3^{r/2}
    \qquad\text{and}\qquad
    \text{maximum weight} \le 10r
\]
for all integer $r\ge 7$.
\end{corollary}

Equivalently, with maximum weight $w=O(r)$, the maximum qubit load must be at least on the order of
\[
    \frac{3^r}{r},
\]
not $3^{r/2}$.

\subsection{Quantum Aaronson--Ambainis conjecture holds for all hamiltonians built on anti-commuting Pauli strings}

We consider $L=3^r$ Pauli strings on  $n\ge \frac{3^r-1}{2}$ qubits,  we also postulate that the weight (degree) of each string is between $r$ and $10r$. The main requirement is that all strings $P_1,\dots, P_L$ pairwisely anti-commute.

There are  many such systems of strings, one of them, constructed by \cite{JKMN} was explained above.

Now we look at hamiltonian
$$
P =\sum_{\ell=1}^L \alpha_\ell P_\ell
$$
We normalize it as follows:
$$
\sum_\ell |\alpha_\ell|^2=1\,.
$$

Notice that  because of normalization
$$
\|P\|_{HS} =1, \quad \|P\|_\op =1\,.
$$
The second equality uses that $P_\ell$ are pairwise anti-commuting. So $P$ is a quantum boolean function, with $\Var[P]=1$.

\medskip

We wish to prove that there are absolute constants $C, K$ such that
\begin{equation}
\label{AAQ}
\Var[P] \le C r^K\, \text{maxInf}[P]\,.
\end{equation}
If we raise $\Var[P]=1$ to the power $K$ we will get the conjectures Aaronson--Ambainis inequality for this (very special) quantum boolean function.
For all scalar boolean functions it is proved in \cite{OSSS} and \cite{Z}.

For every qubit $q$ denote by $S(q,x)$ those $\ell\in [L]$ that have Pauli $X$ in this qubit. Similarly introduce $S(q,y)$, $S(q,z)$.
Now put
\begin{align*}
&\xi_q=\sum_{\ell\in S(q, x)} |\alpha_\ell|^2,
\\
&\eta_q=\sum_{\ell\in S(q, y)} |\alpha_\ell|^2,
\\
&\zeta_q=\sum_{\ell\in S(q, z)} |\alpha_\ell|^2.
\end{align*}

Our set of string $\{P_1,\dots, P_L\}$ is provided with probabilities: the probability to choose $P_\ell$  is $|\alpha_\ell|^2$.
We will be choosing strings independently but with these probabilities.

\medskip

Let us consider event $E(q, x, y)$ that choosing independently two strings one of them has $X$ in qubit $q$ and another $Y$ in qubit $q$.
Similarly consider $E(q, x, z)$ and $E(q, x, z)$.

\medskip

Clearly
\begin{equation}
\label{xyz}
\bP (E(q, x, y)) = \xi_q\cdot \eta_q, \,\,\bP (E(q, x, z)) = \xi_q\cdot \zeta_q,\,\,\bP (E(q, y, z)) = \eta_q\cdot \zeta_q\,.
\end{equation}

As all pairs of strings anti-commute the probability that choosing two different strings we will have for {\bf some} $q$ either $X, Y$ overlap, or $X, Z$ overlap or $Y, Z$ overlap is $1$.

This means that
$$
\sum_q \big( \bP (E(q, x, y)) + \bP (E(q, x, z)) + \bP (E(q, y, z))\big) \ge 1\,.
$$
Hence,
$$
1\le \sum_q \big(\xi_q\cdot \eta_q + \xi_q\cdot \zeta_q + \eta_q\cdot \zeta_q\big)
$$
Put
$$
D(q) = \xi_q + \eta_q+\zeta_q\,.
$$
Then 
$$
1\le \sum_q \frac12 ( D(q)^2 - \xi_q^2-\eta_q^2 -\zeta_q^2) \le \frac13 \sum_q D(q)^2\,.
$$
Automatically
$$
3\le \max_q D(q) \cdot \sum_q D(q)\le \max_q D(q) \sum_q \sum_{\ell\in NT(q)} |\alpha_\ell|^2\,.
$$
Here $NT(q)$ is the union of indices of strings being non-trivial over $q$, that is $S(q, x)\cup S(q, y)\cup S(q, z)$.

\medskip

Using Fubini we write
$$
3\le \max_q D(q) \,\,\sum_{\ell=1}^L |\alpha_\ell|^2 \cdot \sharp \{q: \, P_\ell \,\, \text{has a non-trivial Pauli on qubit}\,\,q\}\,.
$$ 
But this sum over $q$ is less than $10r$ by the assumption on the weight of each string is at most $10\,r$.

\medskip

Therefore,

\begin{equation}
\label{Dq}
\max_q D(q) \ge \frac{3}{10r}\,.
\end{equation}
But 
$$
D(q) =\xi_q+\eta_q+\zeta_q = \sum_{\ell \,\, \text{that has non-trivial Pauli over}\,\,q} |\alpha_\ell |^2 = \Inf_q[P]\,.
$$

We got that
$
\frac{3}{10 r} \le \max_q D(q) =\text{maxInf}[P]\,.
$
This implies \eqref{AAQ} with $C=10$, $K=1$.

\begin{remark}
\label{ACP}
If we relax the requirement on Pauli strings in term of weight, for example by saying that each anti-commuting string can have degree between $r$ an $r^{100}$, then we still have \eqref{AAQ}.
Looks like the more non-commuting are strings the better is Aaronson--Ambainis inequality. At this moment this is a philosophical statement.
\end{remark}

\subsection{Why $p=\frac{2d}{d+1}$ exponent in commutative Bohnenblust--Hille inequality is sharp}
\label{exponent}

There is a Kahane--Salem--Zygmund random construction (see e.g \cite{D}, 
Sections 7.3, 7.4) where this is proved. Now we suggest a non-random way to see this.
Let us fix $d$ and large $L=2^k$, and let us split $dL$ qubits into $d$ blocks of $L$ qubit each. 
Operator $Z_{j, \ell}$ has identity on all qubits except the $\ell$-th qubit in 
$j$-th block ($j\in [d], \ell\in [L]$) where we put $Z$. 

Operator $H$ is the Sylvester--Hadamard matrix of size $L\times L$= $2^k\times 2^k$.

Consider $B_\ell$ from \eqref{eq:B-expanded}. Operator 
$$
S=B_1+\dots +B_L
$$
has  $L^{d-1}$ different strings, in fact,
\begin{equation}
\label{B}
B =L\cdot Z_{1, \vec 0}  \sum_{a_2\dots, a_d\in \{0,1\}^k} (-1)^{a_{2}\cdot a_3}\dots (-1)^{a_{d-1}\cdot a_d} Z_{2, a_2}\dots Z_{d, a_d}
\end{equation}
 As we have $L^{d-1}$ string of absolute value $L$ we have
$$
\|\hat B\|_{\ell^p} = (L^p\cdot L^{d-1})^{1/p} = L^{\frac{d-1}{p}+1}\,.
$$
If we look at the sum in \eqref{B}, we see that this is actually the sum of $L^2$ blocks of size $2^{dL}\times 2^{dL}$ of the block-matrix $D_2H\dots HD_d$. 

It is easy to see that this is exactly  a product that involves $d-2$ factors $H$.

From this and \eqref{B} it easy to see (by the same doubling trick as in Subsection \ref{sumOFsquares}) that 
$$
\|B\|_{op}  \le L\,L^{\frac{d-1}2} L^{\frac12}L^{\frac 12} = L^{\frac{d}2+1}\,.
$$
Looking at BH inequality $\|\hat B\|_p \le C(d) \|B\|_\op$ that becomes $L^{\frac{d-1}{p}+1} \le C(d) \,  L^{\frac{d}2+1}$ and choosing $L\to \infty$ we see that
$$
p\ge \frac{2(d-1)}{d}\,.
$$

But the degree of $B$ is $d-1$, so, we proved the estimate in the title of Subsection \ref{exponent}.

We used only one Pauli, namely, Z, so it is actually a commutative construction.


\begin{thebibliography}{999}




\bibitem{AA}
S. Aaronson and A. Ambainis,
\emph{The need for structure in quantum speedups},
Theory Comput. 10 (2014), 133--166.





\bibitem{BGKT}
Sergey Bravyi, David Gosset, Robert K\"onig, and Kristan Temme.
\newblock Approximation algorithms for quantum many-body problems.
\newblock {\em J. Math. Phys.}, 60(3):032203, 18, 2019.

\bibitem{BuEtAl}
K.~Bu, R.~J. Garcia, A.~Jaffe, D.~E. Koh, and L.~Li,
\emph{Complexity of quantum circuits via sensitivity, magic, and coherence},
Commun. Math. Phys. \textbf{405} (2024), article 161,
\href{https://doi.org/10.1007/s00220-024-05030-6}{doi:10.1007/s00220-024-05030-6}.

\bibitem{BLP}
L. Ben Efraim, F. Lust-Piquard,
\newblock Poincar\'e type inequalities on the discrete cube and in the CAR algebra.
\newblock Probability Theory and Related Fields 141 (2008), no. 3--4, 569--602.
DOI: 10.1007/s00440-007-0094-x.

\bibitem{BSVZ}
L. Becker, J. Slote, A. Volberg, H. Zhang,
\newblock Approximating the operator norm of local Hamiltonians via few quantum states,
\newblock arXiv:2509.11979v3, pp. 1--34.

\bibitem{BH}
H. F. Bohnenblust, E.  Hille,
\newblock On the Absolute Convergence of Dirichlet Series.”
\newblock Annals of Mathematics (Second Series) 32 (1931), no. 3, 600--622. 


\bibitem{CHP}
Chen, S., Huang, H.-Y., Preskill, J. (2023). 
\newblock Learning to Predict Arbitrary Quantum Processes. 
\newblock PRX Quantum, 4, Article 040337. https://doi.org/10.1103/prxquantum.4.040337



\bibitem{D} 
A. Defant, D. Garcia, M. Maestre, P. Sevilla-Peris
\newblock Dirichlet Series And Holomorphic Functions In High Dimensions,
\newblock Cambridge University Press, 2019, ISBN 978-1-108-47671-3 Hardback


\bibitem{DFOOS}
A. Defant, L. Frerick, J. Ortega-Cerd\`a, M. Ounaies, and K. Seip, 
\newblock The Bohnenblust--Hille inequality for homogeneous polynomials is hypercontractive. 
\newblock Ann. Math. (2),
174(1):485–497, 2011.

\bibitem{DMP}
 A. Defant, M. Mastylo, and A. P\'erez,
\newblock On the Fourier spectrum of functions on
boolean cubes. Mathematische Annalen, 374(1–2):653--680, 2019.

\bibitem{DR}  H. A. Dye, B. Russo, 
\newblock A note on unitary operators in \(C^{*}\)-algebras,
\newblock Duke Mathematical Journal, Vol. 33, No. 2, pp. 413–416 (1966).


\bibitem{EI}
A. Eskenazis, P. Ivanisvili,
\newblock Learning Low-Degree Functions from a Logarithmic Number of Random Queries.
\newblock Proceedings of the 54th Annual ACM SIGACT Symposium on Theory of Computing4(STOC 2022), Rome, Italy, June 20--24, 2022, pp. 203–207.5DOI: 10.1145/3519935.3519981.

\bibitem{EPR}
A. Einstein, B. Podolsky, and N. Rosen, 
\newblock Can Quantum-Mechanical Description of Physical Reality Be Considered Complete?
\newblock Phys. Rev. 47, 777 (1935).



\bibitem{GHZ}
Daniel M. Greenberger, Michael A. Horne, and Anton Zeilinger, 
\newblock in "Bell's Theorem, Quantum Theory and Conception of the Universe", 
\newblock edited by M. Kafatos (Kluwer Academic, Dordrecht, 1989), p. 69.

\bibitem{HKP}
H.-Y. Huang, R. Kueng, J. Preskill, (2020). 
\newblock Predicting many properties of a quantum system from very few measurements. 
\newblock Nature Physics, 16, 1050–1057. https://doi.org/10.1038/s41567-020-0932-7.

\bibitem{JKMN} Z.~Jiang, A.~Kalev, W.~Mruczkiewicz, and H.~Neven,
\emph{Optimal fermion-to-qubit mapping via ternary trees with applications to reduced quantum states learning},
Quantum \textbf{4} (2020), 276,
\href{https://doi.org/10.22331/q-2020-06-04-276}{doi:10.22331/q-2020-06-04-276}.


\bibitem{L}
Elliott~H. Lieb.
\newblock The classical limit of quantum spin systems.
\newblock {\em Comm. Math. Phys.}, 31:327--340, 1973.


\bibitem{Mermin1990} 
D. Mermin,
\newblock Extreme Quantum Entanglement in a Superposition of Macroscopically Distinct States,
\newblock Phys. Review Letters, v. 65, no.15 (1990), 1838--1840.

\bibitem{MO}
A. Montanaro, T. J. Osborne
\newblock Quantum Boolean Functions.
\newblock Chicago Journal of Theoretical Computer Science,42010, Article 1, pp. 1–45.5DOI: 10.4086/cjtcs.2010.001.

\bibitem{OSSS}
R. O'Donnell, M. Saks, O. Schramm,  R.A. Servedio, 
\newblock Every decision tree has an influential variable.
\newblock Proceedings of the 46th Annual IEEE Symposium on Foundations4of Computer Science (FOCS 2005), pp. 31--39.

\bibitem{RWZ}
C.~Rouz\'e, M.~Wirth, and H.~Zhang,
\emph{Quantum Talagrand, KKL and Friedgut's theorems and the learnability of quantum Boolean functions},
Commun. Math. Phys. \textbf{405} (2024), article 95,
\href{https://doi.org/10.1007/s00220-024-04981-0}{doi:10.1007/s00220-024-04981-0}.

\bibitem{GS}
G. Schechtman,
\newblock Entropy versus influence for complex functions of modulus one,
\newblock arXiv:2009.12753 [math.CO], 2020.

\bibitem{S}
J.~Slote,
\newblock Dense Hamiltonians at the Parseval Limit,
\newblock arXiv:2608.01424, pp.1--8.

\bibitem{T}
J. Tomiyama, 
\newblock On the projection of norm one in \(C^{*}\)-algebras.
\newblock Proceedings of the Japan Academy, Series A, Mathematical Sciences, 33(10), 608-612 (1957).

\bibitem{VZ} 
A. Volberg, H. Zhang, 
\newblock Noncommutative Bohnenblust–Hille inequalities.
\newblock Mathematische Annalen 389 (2024), no. 2, 1609--1652.4DOI: 10.1007/s00208-023-02680-0.

\bibitem{VZ1}
A. Volberg, H. Zhang
\newblock Two tensor Rudin--Shapiro constructions,
\newblock Preprint, 2026, pp. 1--8.

\bibitem{Z}
H. Zhang,
\newblock
The Boolean case of the Aaronson–Ambainis influence conjecture
without OSSS.
\newblock Preprint, 2026, pp. 1--5.





\end{thebibliography}
\end{document}